\documentclass[11pt,reqno]{amsart}
\usepackage{amsmath}
\usepackage{amssymb}
\usepackage{amsthm}
\usepackage{enumerate}
\usepackage[usenames]{color}
\usepackage{eepic,epic}
\usepackage{url} 
\usepackage{multirow}
\usepackage{epstopdf}
\usepackage{makecell}
\usepackage{graphicx}
\usepackage{comment}

 \usepackage{algorithm,algorithmic}

\newtheorem{thm}{Theorem}[section]

\newtheorem{lem}[thm]{Lemma}
\newtheorem{prop}[thm]{Proposition}

\theoremstyle{definition}

\theoremstyle{remark}
\newtheorem{rem}[thm]{Remark}
\numberwithin{equation}{section}

\newcommand{\vertiii}[1]{{\left\vert\kern-0.25ex\left\vert\kern-0.25ex\left\vert #1 
    \right\vert\kern-0.25ex\right\vert\kern-0.25ex\right\vert}}

\begin{document}
\title[]
{On the convergence result of the gradient-push algorithm on directed graphs with constant stepsize} 
\author{Woocheol Choi}
\address{Department of Mathematics, Sungkyunkwan University, Suwon 440-746, Republic of Korea}
\email{choiwc@skku.edu}

\author{Doheon Kim}
\address{{Department of Mathematical Data Science, Hanyang University, Hanyangdaehak-ro 55, Sangnok-gu, Ansan, Gyeonggi-do 15588, Republic of Korea}}
\email{doheonkim@hanyang.ac.kr}

\author{Seok-Bae Yun}
\address{Department of Mathematics, Sungkyunkwan University, Suwon 440-746, Republic of Korea}
\email{sbyun01@skku.edu}


\maketitle

\begin{abstract}
Distributed optimization has recieved a lot of interest due to its wide applications in various fields. It consists of multiple agents that connected by a graph and optimize a total cost in a collaborative way. Often in the applications, the graph of the agents is given by a directed graph. The gradient-push algorithm is a fundamental method for distributed optimization for which the agents are connected by a directed graph. Despite of its wide usage in the literatures, its convergence property has not been established well for the important case that the stepsize is constant and the domain is the entire space. 
This work proves that the gradient-push algorithm with stepsize $\alpha>0$ converges exponentially fast to an $O(\alpha)$-neighborhood of the optimizer  if the stepsize $\alpha$ is less than a specific value. For the result, we assume that each cost is smooth and the total cost is strongly convex. Numerical experiments are provided to support the theoretical convergence result. \textcolor{black}{We also present a numerical test showing that the gradient-push algorithm may approach a small neighborhood of the minimizer faster than the Push-DIGing algorithm which is a variant of the gradient-push algorithm involving the communication of the gradient informations of the agents.}
\end{abstract}

\section{Introduction}

In this paper, we are concerned with the distributed optimization:
\begin{equation}\label{goal}
\min_{x \in \mathbb{R}^d} f(x) = \sum_{i=1}^n f_i (x).
\end{equation}
Here, $n$ denotes the number of agents in the network, and $f_i: \mathbb{R}^d \rightarrow \mathbb{R}$ represents the local objective function available exclusively to agent $i$, for each $1 \leq i \leq n$. {Each agent can directly communicate with its neighboring agents in the network, which can be represented as a graph.}

Distributed optimization has been studied extensively in recent years due to its various applications, including wireless sensor networks \cite{BG, SGRR}, multi-agent control \cite{MC1, MC2, MC3}, smart grids \cite{SG1, SG2}, and machine learning \cite{ML1, ML2, ML3, ML4, ML5}. There has been a lot of works  for distributed optimization algorithms. We refer to \cite{NO1, NOS, SLWY, SLYWY}, and references therein. A fundamental algorithm is the distributed gradient descent algorithm (DGD) introduced in \cite{NO1}. There are also various versions of distributed optimization algorithms, such as EXTRA \cite{SLWY}, {decentralized gradient tracking} \cite{NOS, QL, XK2}. Recently, there has been interest in designing {communication-efficient algorithms} for distributed optimization \cite{S, CO, JXM, BBKW}. 

Often in practice,  the graph between the agents is described by a directed network since the communication between the agents are not bidirectional due to the heterogeneity of their communication powers. A fundamental algorithm for the directed graph is so called the gradient-push algorithm introduced by Nedic {and} Olshevsky \cite{NO1, NO2}. It makes use of the push-sum algorithm for consensus \cite{KDG, TLR2} on directed graphs. The gradient-push algorithm has been widely used by numerous researchers and its variants have been proposed in the literatures. For example, the algorithm has been applied for online distributed optimization \cite{AGL, LGW,O3}.  The work \cite{Q2} studied the gradient-push algorithm involving the quantized communications, and the work \cite{CK2} designed the gradient-push algorithm with the event-{triggered} communication.
The gradient-push algorithm has also been utilized for constrained distributed optimization \cite{C1} and zeroth-order distributed optimization \cite{GF}.  We also refer to \cite{AGP} for asynchronous gradient-push algorithm.  Variants of the gradient-push algorithm involving the communication of gradient information are presented in  \cite{NOS, XK, XXK, XK2}.

Despite its wide applications, the convergence analysis of the gradient-push algorithm has not been established well so far when compared to the DGD \cite{NO1} which is the counterpart of the gradient-push for undirected graph. In particular, the previous works \cite{NO1, NO3, NO2, TT} assumed the boundedness of either  the state variable or the gradient for the convergence analysis. 


Before discussing the previous results and our results, we recall that the gradient-push algorithm \cite{NO1} is described as follows: Given inital data $x_i (0) \in \mathbb{R}^d$, {$y_i (0)=1$ for each $1 \leq i \leq n$} and $t \in \mathbb{N} \cup \{0\}$, we update
\begin{equation}\label{eq-2-1}
\begin{split}
w_i (t+1) & = \sum_{j=1}^n W_{ij} {x_j (t),}
\\
y_i (t+1)& = \sum_{j=1}^n W_{ij} {y_j (t),}
\\
z_i (t+1)& = {\frac{w_i (t+1)}{y_i (t+1)},}
\\
x_i (t+1)& = w_i (t+1) - \alpha (t) \nabla f_i (z_i (t+1)),
\end{split}
\end{equation}
where $W_{ij}$ are communication weight such that the matrix $W =\{W_{ij}\}_{1 \leq i,j \leq n}$ is column stochastic and $\alpha (t)>0$ is the stepsize. We refer to Section \ref{sec-2} for the full detail. 

 In the original work \cite{NO2}, Nedic and Olshevsky established the convergence result for convex cost functions when the stepsize $\alpha (t)$ is given as $\alpha (t) = 1/\sqrt{t}$ or satisfies the condition $\sum_{t=1}^{\infty}\alpha (t) =\infty$ and $\sum_{t=1}^{\infty}\alpha (t)^2 <\infty$. In \cite{NO3}, the convergence result was obtained when the cost function is strongly convex, and $\alpha (t) = c/t$ for some $c>0$. Recently, the work \cite{TT} established a convergence result for non-convex cost functions. The results are summarized in Table \ref{known results1}. It is worth mentioning that the convergence results in the previous works \cite{NO1, NO2, TT} assumed that the {gradients} of the cost functions are uniformly bounded.

We mention that if $W$ is given by a doubly-stochastic matrix, then $y_i (t)=1$ for all $t \geq 0$ and $1\leq i \leq n$. The algorithm \eqref{eq-2-1}  is then equal to the distributed gradient descent (DGD) algorithm  (\cite{NO1}):
\begin{equation}\label{eq-1-1}
x_i (t+1) = \sum_{j=1}^n W_{ij} x_j (t) - \alpha (t) \nabla f_i (x_i (t)).
\end{equation}
The work \cite{YLY} studied the convergence property of \eqref{eq-1-1} when the stepsize is given by a constant $\alpha (t)\equiv \alpha$. The authors showed that if the cost functions are smooth and convex, then the sequence of \eqref{eq-1-1} is uniformly bounded when $\alpha >0$ is smaller than a constant determined by the communication network and the smoothness constant. The boundedness result was also obtained in the recent work \cite{CK} with replacing the convexity assumption on each cost function $f_j$ by the strongly convexity on the total cost \mbox{function $f$.} In addition, the works \cite{YLY, CK} showed that the sequence of \eqref{eq-1-1} converges to an $O(\alpha)$-neighborhood of the optimizer whenever the stepsize $\alpha >0$ is less than a specific value.

Having mentioned the above results, one may ask the following  questions on the gradient-push algorithm \eqref{eq-2-1}:

\medskip

\noindent \emph{\textbf{Question 1}: Can we show that the sequence of \eqref{eq-2-1} is uniformly bounded without assuming  either the boundedness of the state variable or the gradient?}

\medskip

\noindent \emph{\textbf{Question 2}: Does the gradient-push algorithm with constant stepsize $\alpha >0$ converge to an $O(\alpha)$-neighborhood of the optimizer of \eqref{goal}?}

\medskip

The aim of this paper is to address these  questions. Precsisely, we will show that the sequence of \eqref{eq-2-1} is bounded if the stepsize $\alpha (t)\equiv \alpha >0$ is less than a specific value when  each cost function is $L$-smooth and the total cost is strongly convex. Under these assumptions, we will also show that the gradient-push algorithm with a constant stepsize $\alpha (t) \equiv \alpha >0$ converges exponentially fast to an $O(\alpha)$-neighborhood of the optimizer $x_*$ provided that $\alpha \leq \alpha_0$ for some $\alpha_0 >0$ determined the properties of the cost functions and the communication network. 

{ 
Given the above result, we recall that variants of the gradient-push algorithm that incorporate the exchange of gradient information with constant step sizes have been shown to converge linearly to the minimizer of \eqref{goal}, provided the step sizes fall within a specific range \cite{NOS, PSXN,  XK, XK2, XXK}.  DEXTRA \cite{XK} is an extension of EXTRA \cite{SLWY} to the directed graph and it
requires that the stepsize satisfies a lower and an upper bound for the convergence. The Push-DIGing/ADD-OPT algorithm \cite{NOS, XXK} can be seen as an extension of the gradient-push algorithm \eqref{eq-2-1} with communication of the gradient information.  The Push-DIGing/ADD-OPT algorithm requires only an upper bound on the stepsize for the convergence and exhibit a faster convergence than DEXTRA in the numerical test (\cite{NOS, XXK}). We also refer the works \cite{PSXN, XK2} on the push–pull/AB algorithm that utilize both row-stochastic
matrix and column-stochastic matrix simultaneously and achieve exact linear convergence.

When comparing the gradient-push algorithm with its variant, the Push-DIging algorithm, the advantage of the Push-DIging algorithm is that it exhibits precise linear convergence. On the other side, the gradient-push algorithm may approach a small neighborhood of the minimizer faster than the Push-DIGing algorithm when we compare the algorithms with stepsizes maximizing the performance (refer to the numerical test in Section \ref{sec-8}). This observation motivates the consideration of a hybrid algorithm that employs  the gradient-push for a certain number of iterations before switching to the Push-DIGing algorithm (see \cite{CCK}). The efficiency of the gradient-push algorithm becomes more evident when we evaluate performance in terms of communication cost, as Push-DIGing also requires communication of gradient tracking variables. A detailed discussion will be provided in Section \ref{sec-8}. The aforementioned advantages of the gradient-push algorithm support the importance of  establishing its convergence property from a practical perspective.}

In order to achieve the properties of the gradient-push algorithm, we consider two measures:
\begin{itemize}
\item A distance between the average $\bar{w}(t)$ of $\{w_i (t)\}_{i=1}^N$ and the optimal point $w_*$ for the total cost $f$. 
\item A measurement of $\{w_i (t)\}_{i=1}^N$ regarding a weighted consensus associated to the communication matrix $W$. 
\end{itemize}
We will derive two sequential estimates of the above measures. By analyzing the sequential estimates, we will achieve the boundedness and the convergence properties of the gradient-push algorithm (refer to Sections \ref{sec-3}-\ref{sec-7}). Our proofs provide a flexible framework that can be extended to  numerous variants of the gradient push algorithm. 
For example, we believe the it can be applied for  the decentralized TD learning on directed graph \cite{LL} whose anlaysis is an open problem.
In addition, the framework can be used for obtaining new sharp estimates for the variants containing the online gradient-push algorithm \cite{AGL, LGW},  the  gradient-push algorithm involving the quantized communications  \cite{Q2}, and the zeroth-order  gradient-push algorithm \cite{GF}.


  
 This paper is organized as follows: In Section 2, we review the push-sum algorithm and related properties. In Section 3, we state the main convergence results of this paper. In Section 4, we obtain
 an estimate measuring the effect of the weighted consensus process. In Section 5, we derive a sequential
 estimate on the distance between an average  of the states and the optimizer. Combining the estimates
 from Sections 4 and 5,  we prove the uniform boundedness of the sequence in Section 6 and establish the main convergence result in Section 7. Finally, in Section \ref{sec-8}, we provide numerical results supporting our theoretical findings. In addition, we compare the performance of the  gradient-push, Push-DIging, and their hybrid algorithms. In  the appendix, we provide an intuition behind the push-sum algorithm for {the readers' understanding.} 
  

\begin{table}[ht]
\centering
\begin{tabular}{|c|c|c|c|c|c|  }\cline{1-6}
& {Convexity}& Regularity & Learning rate & Error  &Rate  
\\ 
\hline
&&&&&\\[-1em]
\cite{NO2} & C  & $\|\nabla f_i \|_{{\infty}}< \infty$ &\makecell{ $\sum_{t=1}^{\infty}\alpha (t) = \infty$\\ $\sum_{t=1}^{\infty}\alpha(t)^2 < \infty$} & $\|x_i (t) -x_*\|$ & $o(1)$  
\\
&&&&&\\[-1em]
\hline
&&&&&\\[-1em]
\cite{NO2} & C  &$\|\nabla f_i \|_{{\infty}}< \infty$& $\alpha (t) = 1/\sqrt{t}$ & $f(\hat{z}_i (t)) - f(z_*)$ & $O\Big( \frac{\log t}{\sqrt{t}}\Big)$  
\\
&&&&&\\[-1em]
\hline
&&&&&\\[-1em]
\cite{NO3} & SC  & $\|z_i (t) \|_{ {\infty}}< \infty$ & $\alpha (t) = \frac{c}{t}$ & $f(\hat{z}_i (t)) - f(z_*)$ & $O\Big(\frac{\log t}{t}\Big)$  
\\
&&&&&\\[-1em]
\hline
&&&&&\\[-1em]
 \cite{TT} & NC &\makecell{$\|\nabla f_i \|_{\infty}<\infty$ \\ $L$-smooth} & \makecell{ $\sum_{t=1}^{\infty}\alpha (t) = \infty$\\ $\sum_{t=1}^{\infty}\alpha(t)^2 < \infty$}  & $\|x_k (t) -x_*\|$& o(1)
\\
&&&&&\\[-1em]
\hline
&&&&&\\[-1em]
 \cite{TT} & NC &\makecell{$\|\nabla f_i \|_{\infty}<\infty$ \\ $L$-smooth} & $\alpha (t) = \frac{c}{t}$ & $\|x_k (t) -x_*\|$& $O\Big(\frac{1}{t}\Big)$
\\
&&&&&\\[-1em]
\hline
&&&&&\\[-1em]
\makecell{This \\ work} & SC& $L$-smooth & $\alpha (t) \equiv \alpha$ & $\|x_k (t) -x_*\|$ & $O(e^{-ct}) + O({\alpha})$  
 \\
\hline
\end{tabular}
\vspace{0.1cm}
\caption{This table summarizes the convergene results for the gradient-push algorithm. Here C (resp., SC) means that the total cost function is assumed to be convex (resp., strongly convex). NC is an abbreviation for nonconvex. Also $x_*$ is an optimizer of \eqref{goal} and  $\hat{z}_i(t)$ is an weighted average of $\{z_i (1), \cdots, z_i (t)\}$. }\label{known results1}
\end{table}

\section{Gradient-push algorithm and related basic results}\label{sec-2}
In this section, we review the gradient-push algorithm  and some preliminary results for obtaining the convergence estimates of the gradient-push algorithm.

We consider a directed graph $G= (V,E)$ with the vertex set $V = \{1,2,\cdots, n\}$ and the edge set $E \subset V\times V$, where the communication between agents are performed. Here $(i,j) \in E$ means that there is a directed edge from agent $i$ to $j$. We define the set of in-neighbors of agent $i$ by
{
\begin{equation*}
N_i^{in} = \{ j \in V~\mid~(j,i) \in E\},
\end{equation*}}
and the set of out-neighbors of agent $i$ {by}
\begin{equation*}
N_i^{out} = \{ j \in V~\mid~ (i,j) \in E\}.
\end{equation*}
Throughout this paper, we assume that $G$ satisfies the following two assumptions:
\begin{itemize}
\item (G1) Every vertex of $G$ has a self-loop, i.e.,
\[
(i,i)\in E\quad(\Leftrightarrow i\in N_i^{in}~\Leftrightarrow~ i\in N_i^{out}),\quad\forall~i\in V.
\]
\item (G2) $G$ is strongly connected, i.e., $G$ contains a directed path from $i$ to $j$ for every $(i,j)\in V\times V$.
\end{itemize}
We set  $d_j = |N_j^{out}|$ and
\begin{equation}\label{eq-2-0}
W_{ij} = \left\{\begin{array}{ll} {\frac{1}{d_j}} & \textrm{if}~j \in N_i^{in}\quad{(\mbox{or equivalently, }i\in N_j^{out}),}
\\
0 & \textrm{if}~j \notin N_i^{in}.
\end{array}\right.
\end{equation}
Let us denote  by $w(t) = \mathrm{col}(w_1 (t), \cdots, w_n (t))\in \mathbb{R}^{nd}$ the column vector obtained by stacking $w_1(t),\dots,w_n(t)$, and similarly for $x(t)\in \mathbb{R}^{nd}$, $y(t)\in \mathbb{R}^d$, and $z(t)\in \mathbb{R}^{nd}$.
Then we may {rewrite} the gradient-push algorithm \eqref{eq-2-1} as
\begin{equation}\label{eq-2-2}
\begin{split}
w(t+1)& = (W\otimes I_d) {x(t),}
\\
y(t+1)& =W {y(t),}
\\
z_i (t+1) & = \frac{w_i (t+1)}{y_i (t+1)},\quad 1 \leq i \leq n,
\\
x(t+1) & = w(t+1) - \alpha \nabla F(z(t+1)),
\end{split}
\end{equation}
where we define $\nabla F(x):= \mathrm{col} (\nabla f_1 (x_1), \cdots, \nabla f_n (x_n))$ for any $x=\mathrm{col}(x_1,\dots,x_n)$, $I_d$ is the $d\times d$ identity matrix, and $\otimes$ is the Kronecker product. 
We assume that   $y_i(0)=1$ for $i=1,\dots,n$. {Note that the positivity of $y(t)$ is ensured by the following inductive argument:
\begin{equation}\label{eq-y-0}
y_i(t+1)=\sum_{j=1}^nW_{ij}y_j(t)\geq W_{ii}y_i(t)=\frac{y_i(t)}{d_i}>0.
\end{equation}}
In addition, the matrix  $W$ is column stochastic, {by the following relation:
\[
\sum_{i=1}^nW_{ij}=\sum_{i=1}^n\frac{1_{N_i^{in}}(j)}{d_j}=\sum_{i=1}^n\frac{1_{N_j^{out}}(i)}{|N_j^{out}|}=1,\quad j=1,\dots,n.
\]
 Note that column stochasticity of $W$ is equivalent to $1_n^\top W = 1_n^\top$. Also note that by \eqref{eq-2-0}, $W_{ij}>0$ if and only if $(j,i)\in E$. By (G2), $W$ is an irreducible matrix, so we can apply the Perron-Frobenius theorem to deduce that $W$ has a  right eigenvector $\pi=(\pi_1, \cdots, \pi_n)^\top$ associated with the eigenvalue 1, of which all entries are positive and satisfy 
 \begin{equation}\label{eq-2-19}
 \sum_{j=1}^n \pi_j =1.
 \end{equation} Moreover, by (G1), $G$ is aperiodic. Since $G$ is both strongly connected and aperiodic, $W$ is primitive, and} it is known that 
\begin{equation}\label{eq-2-35}
W^{\infty}  :=\lim_{k\to\infty}W^k= \pi 1_n^\top
\end{equation}
{with convergence speed geometrically fast \cite{Seneta}.} 
We denote the Euclidean norm by $\|\cdot \|$ and {set the matrix norm $\vertiii{A}$ for a matrix $A \in \mathbb{R}^{n\times n}$} by 
\begin{equation*}
\vertiii{A} = \sup_{x \in \mathbb{R}^n \setminus \{0\}} \frac{\|Ax\|}{\|x\|}.
\end{equation*}
Now we define a weighted inner product as $\langle x, y \rangle_{\pi} = x^\top \textrm{diag}(\pi)^{-1} y$ and its induced norm $\|x\|_{\pi} = \|\textrm{diag}(\sqrt{\pi})^{-1} x\|$. We also set $\vertiii{A}_{\pi}$ as the matrix norm induced by $\|\cdot \|_{\pi}$ for $A \in \mathbb{R}^{n\times n}$, i.e., 
\begin{equation*}
\vertiii{A}_{\pi} = \sup_{x \in \mathbb{R}^n \setminus \{0\}} \frac{\|Ax\|_{\pi}}{\|x\|_{\pi}}.
\end{equation*}
It is easy to check that $\vertiii{A}_{\pi} = \vertiii{\textrm{diag}(\sqrt{\pi})^{-1} A \textrm{diag}(\sqrt{\pi})}$. Also we have
\begin{equation*}
\|\cdot \|_{\pi} \leq \frac{1}{\sqrt{\pi_a}} \|\cdot \|\quad \textrm{and}\quad \|\cdot \| \leq \sqrt{\pi_b} \|\cdot \|_{\pi},
\end{equation*}
where $\pi_a$ and $\pi_b$ denote the minimum and maximum values of $\pi$, { respectively.} We state a basic lemma for $W$.
\begin{lem}
We have
\begin{equation*}
|||W|||_{\pi} = |||W^{\infty}|||_{\pi} = 1.
\end{equation*}
\end{lem}
\begin{proof}
We give a proof of this lemma for {the readers'} convenience. Since $\pi$ is a right eigenvector of $W$, we have $\sum_{j=1}^n w_{ij} \pi_j = w_i$ for all $1 \leq i \leq n$.
Using  this and {the} Cauchy-{Schwarz} inequality, we deduce the following inequality
\begin{equation*}
\begin{split}
\|W x\|_{\pi}^2 & = \sum_{i=1}^n \Big( \sum_{j=1}^n w_{ij} \sqrt{\pi_j} \frac{x_j}{\sqrt{\pi}_j}\Big)^2 \pi_i^{-1}
\\
& \leq \sum_{i=1}^n \Big( \sum_{j=1}^n w_{ij} \pi_j \Big)\Big( \sum_{j=1}^n w_{ij} \frac{x_j^2}{\pi_j} \pi_i^{-1}\Big)
\\
& = \sum_{i=1}^n \sum_{j=1}^n w_{ij} \frac{x_j^2}{\pi_j} = \sum_{j=1}^n \frac{x_j^2}{\pi_j} = \|x\|_{\pi}^2,
\end{split}
\end{equation*}
which proves $|||W|||_{\pi} \leq 1$. Also this directly implies that $|||W^{\infty}|||_{\pi} \leq 1$.  Combining these inequalities with the fact
$W\pi = \pi$ yields that 
\begin{equation*}
|||W|||_{\pi}=|||W^{\infty}|||_{\pi}=1.
\end{equation*}
This finishes the proof.
\end{proof}
We recall the following result from \cite[Lemma 1]{XSKK}:
\begin{lem}[\cite{XSKK}]\label{lem-2-1}
{Under assumptions $(G1)$ and $(G2)$,} we have the estimate $\sigma_W : = |||W- W^{\infty}|||_{\pi} <1$.
\end{lem}
Using this lemma, one may derive the following convergence estimate of $y(t) \in \mathbb{R}^n$ towards $n\pi \in \mathbb{R}^n$:
\begin{equation}\label{eq-2-40}
\|y(t)-n\pi\|_\pi=\left\|(W-W^\infty)^t(1_n-n\pi)\right\|_\pi\leq \sigma_W^t{\|1_n-n\pi\|_\pi,} \quad \forall~ t \geq 0.
\end{equation}
Combining this with \eqref{eq-y-0}, we find that 
\begin{equation}
\min_{1\leq j \leq n} \min_{t \in \mathbb{N}} y_j (t)  >0.
\end{equation}
{ We can similarly define a weighted inner product on $\mathbb R^{nd}$ by 
\[
\langle x, y \rangle_{\pi\otimes 1_d} := x^\top \textrm{diag}(\pi\otimes 1_d)^{-1} y= x^\top (\textrm{diag}(\pi)^{-1}\otimes I_d) y,
\]
and its induced norm and matrix norm can be defined analogously.
\begin{lem}\label{lem-2-2}
For any $n\times n$ matrix $A$, we have
\[
\vertiii{A\otimes I_d}_{\pi\otimes 1_d}=\vertiii{A}_\pi.
\]
\end{lem}
\begin{proof}
We proceed in the following way:
\begin{align*}
\vertiii{A\otimes I_d}_{\pi\otimes 1_d}&=\vertiii{\textrm{diag}(\sqrt{\pi\otimes 1_d})^{-1}( A\otimes I_d) \textrm{diag}(\sqrt{\pi\otimes 1_d})} \\
&=\vertiii{(\textrm{diag}(\sqrt{\pi})^{-1}\otimes I_d)( A\otimes I_d) (\textrm{diag}(\sqrt{\pi})\otimes I_d)} \\
&=\vertiii{[\textrm{diag}(\sqrt{\pi})^{-1} A \textrm{diag}(\sqrt{\pi})]\otimes I_d} \\
&=\vertiii{\textrm{diag}(\sqrt{\pi})^{-1} A \textrm{diag}(\sqrt{\pi})} \\
&=\vertiii{A}_{\pi}.
\end{align*}
In the fourth inequality, we have used that the matrix norm $\vertiii{\cdot}$ is equal to the largest singular value of the matrix.
\end{proof}
}

 

\section{Main convergence result}\label{sec-3}
Throughout this paper, we assume that the cost functions satisfy the following two assumptions.
\begin{itemize}
\item (F1) The gradient of each function $f_i$ is Lipschitz continuous with constant $L_i >0$ for all $1 \leq i \leq d$, i.e., $f_i$ is $L_i$-smooth:
\begin{equation}\label{eq-2-4}
\|\nabla f_i (x) - \nabla f_i (y) \| \leq L_i \|x-y\|\quad \forall~x,y \in \mathbb{R}^d.
\end{equation}
\item (F2) The total cost $f$ is $\beta$-strongly convex with a value $\beta >0$. 
\end{itemize}
 Under the above assumptions, there exists a unique minimizer $w_*$ of the cost $f$. 
In order to analyze the convergence property of the algorithm \eqref{eq-2-1}, we consider the following two {quantities:}
\begin{equation*}
A_t:= \|\bar{w}(t) -w_*\|\quad \textrm{and}\quad B_t:=\|w(t) - (W^{\infty}\otimes I_d) w(t)\|_{\pi}.
\end{equation*}
Here $A_t$ measures the distance between the average of states $\bar{w}(t)$ and the optimal point $w_*$. Also  we note that
\begin{equation}\label{eq-2-15-a}
(W^{\infty}\otimes I_d) w(t) =(\pi 1_n^T \otimes I_d) w(t) = n\pi \otimes \bar{w}(t).
\end{equation}
Therefore $B_t$ measures how {close $w(t)$ is} to the weighted consensus states $n\pi \otimes \bar{w}(t)$.
Obtaining convergence properties of $A_t$ and $B_t$ will give a convergence estimate of the sequence $w(t)$ to $n\pi \otimes w_*$ in view of the following lemma.
\begin{lem} For all $t \geq 0$, we have
\begin{equation*}
\|w(t) - n\pi \otimes w_*\|_{\pi\otimes 1_d} \leq \|w(t) - n\pi \otimes \bar{w}(t)\|_{\pi\otimes 1_d} +n \|\bar{w}(t) -w_*\|.
\end{equation*}
\end{lem}
\begin{proof}
We apply the triangle inequality to find
\begin{equation}\label{eq-2-15}
\|w(t) - n\pi \otimes w_*\|_{\pi\otimes 1_d} \leq \|w(t) - n\pi \otimes \bar{w}(t)\|_{\pi\otimes 1_d} + \|n\pi \otimes (\bar{w}(t) - w_*)\|_{\pi\otimes 1_d}.
\end{equation}
By \eqref{eq-2-19},
\begin{equation*}
\|n\pi \otimes (\bar{w}(t) -w_*)\|_{\pi\otimes 1_d}^2 = n^2\sum_{j=1}^n \pi_j \|\bar{w}(t) -w_*\|^2 = n^2\|\bar{w}(t) -w_*\|^2.
\end{equation*}
Combining this with \eqref{eq-2-15} gives
\begin{equation*}
\|w(t) -n \pi \otimes w_*\|_{\pi\otimes 1_d} \leq \|w(t) - n\pi \otimes \bar{w}(t)\|_{\pi\otimes 1_d} + n\|\bar{w}(t) -w_*\|.
\end{equation*}
The proof is done.
\end{proof}
 
In order to state the main convergence results, we set the following values: 
\begin{equation}\label{eq-2-43}
\delta = \max_{t \in \mathbb{N}_0} \max_{1\leq i \leq n} \frac{1}{y_i (t)}\quad \textrm{and}\quad \rho = \sigma_W (1+\alpha L \delta),
\end{equation}
where $\sigma_W  <1$ by Lemma \ref{lem-2-1}. 
We also set $\gamma = \frac{\beta L}{\beta + L}$ and 
\begin{equation*}
\begin{split}
\mathbf{D}_1  &=  \delta {\|1_n-n\pi\|_\pi  } + \sqrt{\sum_{j=1}^n \frac{1}{\pi_j}} 
\\
\mathbf{D}_2 & =  L \delta {\|1_n-n\pi\|_\pi   }   \|w_*\|+\|\nabla F(1_n\otimes w_*)\|_{\pi\otimes1_d}.
\end{split}
\end{equation*}  
In the following result, we establish {boundedness of the sequence $\{(A_t, B_t)\}_{t \geq 0}$} under a condition on the stepsize $\alpha >0$.
\begin{thm}\label{thm-2-10}
 Let $b = \frac{n\gamma}{4L \delta}$.  Assume that (F1) and (F2) hold, and that $\alpha>0$ satisfies 
\begin{equation}\label{eq-4-20}
\alpha \leq \frac{2}{L+\beta} \quad \textrm{and}\quad \alpha < \frac{b(1-\sigma_W)}{L \sigma_W (\delta b + \delta {\|1_n-n\pi\|_\pi} + \sqrt{\sum_{j=1}^n \frac{1}{\pi_j}})} 
\end{equation}
 for all $t\geq0$. Then the sequence $\{(A_t, B_t)\}_{t \in \mathbb{N}_0}$ is uniformly bounded in time, i.e.,
\[
\sup_{t\geq0}A_t<\infty,\quad \sup_{t\geq0}B_t<\infty.
\]
More precisely, take a value $t_0 \in \mathbb{N}$ such that   $ (L\delta/n) {\|1_n-n\pi\|_\pi \sigma_W^{t_0} } \leq \gamma   /2$ and define $R>0$ as
\begin{equation*}
R = \max\Big\{ A_{t_0},~\frac{B_{t_0}}{b}, ~2\|w_*\|,~ \frac{\sigma_W \alpha \textbf{D}_2}{b(1-\sigma_W) - \alpha L \sigma_W (\delta b +\delta {\|1_n-n\pi\|_\pi   } + \sqrt{\sum_{j=1}^n \frac{1}{\pi_j}})}\Big\}.
\end{equation*}
Then we have
\begin{equation*}
A_t \leq R \quad \textrm{and} \quad B_t \leq bR
\end{equation*}
for all $t \geq t_0$.
\end{thm}
We mention that the above boundedness result may hold for a larger range of $\alpha >0$ than that guaranteed by Theorem \ref{thm-2-10}. We refer to \cite{YLY, CK} for the discussion of this issue for the decentralized gradient method on undirected {graphs.} 

Now we state the main convergence result for the gradient-push algorithm.
\begin{thm}\label{thm-2-5}   Assume that (F1) and (F2) hold, and that  $\alpha >0$ satisfies $\alpha \leq \frac{2}{L+\beta}$ and the sequence $\{A_t\}$ is bounded, i.e., there exists $R>0$ such that $A_t \leq R $ for all $t \geq 0$.  Set $\bar{R} = L\textbf{D}_1 R +\textbf{D}_2>0$. Then  for all $t \geq 0$ we have
\begin{equation}\label{eq-2-16}
\|w(t) -n \pi \otimes \bar w(t)\|_{\pi\otimes 1_d} \leq \rho^t \|w(0) -n \pi \otimes \bar w(0)\|_{\pi\otimes 1_d} + \frac{\alpha \sigma_W \bar{R}}{1-\rho}
\end{equation}
and
\begin{equation*}
\begin{split}
\|\bar{w}(t) -w_*\| & \leq (1-\gamma \alpha )^t \|\bar{w}(0) -w_*\|+ \frac{\alpha L\delta}{n} \|w(0) -n \pi \otimes \bar w(0)\|_{\pi\otimes 1_d}t\Big[ \max\{1-\gamma \alpha, \rho\}\Big]^{t-1}
\\
&\quad + \frac{\alpha L \delta}{n} {\|1_n-n\pi\|_\pi}  (\|w_*\|+R) t\Big[ \max\{1-\gamma \alpha, e^{-c}\}\Big]^{t-1} + \frac{\alpha L\delta}{n\gamma} \frac{\sigma_W \bar{R}}{1-\rho}.
\end{split}
\end{equation*}
\end{thm}
The above result establishes {exponentially fast convergence of $w(t)$ to an $O\Big( \frac{\alpha}{1-\rho}\Big)$ neighborhood of $n\pi \otimes \bar{w}_*$.}  In order to prove the above two results, we will derive two sequential inequalities of $A_t$ and $B_t$ in the next two sections. Based on those inequalities, we prove Theorem \ref{thm-2-10} in Section \ref{sec-6}, and then {prove} Theorem \ref{thm-2-5} in Section \ref{sec-8}.

\section{Weighted Consensus}\label{sec-4}
In this section, we give an estimate of $\|w(t+1) - (W^{\infty}\otimes I_d) w(t+1)\|_{\pi\otimes 1_d}$ in terms of $\|w(t) - (W^{\infty}\otimes I_d)w(t)\|_{\pi\otimes 1_d}$ and $\|\bar{w}(t) -w_*\|$. As a preliminary step, we bound a measure of the consensus for the variable $z(t)$ in terms of $w(t)$ in the following lemma.
\begin{lem}\label{lem-2-5} We have the following estimate
\begin{equation*} 
 \| z(t) - 1_n\otimes\bar{w}(t) \|_{\pi\otimes1_d}   \leq \delta \|w(t) -n \pi \otimes \bar{w}(t)\|_{\pi\otimes1_d} + \delta {\|1_n-n\pi\|_\pi \sigma_W^{t} }\Big( \|\bar{w}(t)-w_*\|+\|w_*\|\Big). 
\end{equation*}
for all $t \in \mathbb{N}\cup\{0\}$ and $1 \leq i \leq n$.
\end{lem}
\begin{proof}
Using the definition of $z(t)$ in \eqref{eq-2-1} and $\delta >0$ from \eqref{eq-2-43},  we find
\begin{equation*}
\begin{split}
 \| z(t) - 1_n\otimes\bar{w}(t)\|_{\pi\otimes1_d}& = \Big[ \sum_{j=1}^n \frac{1}{\pi_j} \Big( \frac{w_j (t)}{y_j (t)} -\bar{w}(t)\Big)^2 \Big]^{1/2}
\\
&\leq \delta \Big[ \sum_{j=1}^n \frac{1}{\pi_j} (w_j (t) - y_j (t) \bar{w}(t))^2\Big]^{1/2}
\\
&=\delta \|w(t) -y(t) \otimes \bar{w}(t)\|_{\pi\otimes1_d}.
\end{split}
\end{equation*}
By applying the triangle inequality here, we get
\begin{equation}\label{eq-2-31}
\| z(t) - 1_n\otimes\bar{w}(t)\|_{\pi\otimes1_d}\leq \delta \Big(\|w(t) - n\pi \otimes \bar{w}(t)\|_{\pi\otimes1_d} + \|(n\pi -y(t)) \otimes \bar{w}(t)\|_{\pi\otimes1_d}\Big).
\end{equation}
 Here we use \eqref{eq-2-40} to estimate
\begin{equation}\label{eq-2-32}
\begin{split}
\|(n\pi -y(t)) \otimes \bar{w}(t)\|_{\pi\otimes 1_d}^2 &= \sum_{j=1}^n \frac{1}{\pi_j} \Big( n\pi_j- y_j(t)\Big)^2 \|\bar{w}(t)\|^2
\\
&=\|n\pi-y(t)\|_\pi^2\|\bar w(t)\|^2\\
&\leq   \|1_n-n\pi\|_\pi^2 \sigma_W^{2t} \|\bar{w}(t)  \|^2.
\end{split}
\end{equation} 
Putting this in \eqref{eq-2-31} yields the following estimate
\begin{equation*}
\begin{split}
\| z(t) - &1_n\otimes\bar{w}(t)\|_{\pi\otimes1_d} \leq \delta \|w(t) -n \pi \otimes \bar{w}(t)\|_{\pi\otimes 1_d} + \delta {\|1_n-n\pi\|_\pi \sigma_W^{t} }\|\bar{w}(t)\|
\\
&\leq \delta \|w(t) - n\pi \otimes \bar{w}(t)\|_{\pi\otimes 1_d} + \delta {\|1_n-n\pi\|_\pi \sigma_W^{t} }\Big( \|\bar{w}(t)-w_*\|+\|w_*\|\Big).
\end{split}
\end{equation*}
The proof is done.
\end{proof}
We are ready to establish the main estimate of this {section.}
\begin{prop}\label{lem-3-2}
Assume that  (F1) and (F2) hold, and that $\kappa(t)\geq1$ for all $t\geq0$. Then we have \begin{equation}\label{eq-3-1}
\begin{split}
&\|w(t+1) -n\pi \otimes \bar{w}(t+1) \|_{\pi\otimes 1_d} 
\\ 
&\leq  \sigma_W(1+\alpha L\delta)\|w(t)- n\pi \otimes \bar{w}(t) \|_{\pi \otimes 1_d}  + \alpha L\sigma_W D_1[t] \|\bar{w}(t) - w_*\|
+\alpha\sigma_W D_2[t] ,
\end{split}
\end{equation}
where 
\begin{equation*}
\begin{split}
D_1 [t] &=  \delta {\|1_n-n\pi\|_\pi \sigma_W^{t} } + \sqrt{\sum_{j=1}^n \frac{1}{\pi_j}} 
\\
 D_2 [t]& =  L \delta {\|1_n-n\pi\|_\pi \sigma_W^{t} }   \|w_*\|+\|\nabla F(1_n\otimes w_*)\|_{\pi\otimes1_d}.
\end{split}
\end{equation*} 
\end{prop}
\begin{proof}
First we prove \eqref{eq-3-1}. Using the definition \eqref{eq-2-2} and the relation \eqref{eq-2-15-a}, we proceed as {follows:}
\begin{equation}\label{eq-3-10}
\begin{split}
&\|w(t+1) -(W^{\infty}\otimes I_d) w(t+1) \|_{\pi\otimes 1_d}\\
&=\|((I_n-W^{\infty})\otimes I_d) w(t+1) \|_{\pi\otimes 1_d} 
\\
& = \|[((I_n-W^{\infty})W)\otimes I_d] x(t)\|_{\pi\otimes 1_d}
\\
&= \|[((I_n-W^{\infty})W)\otimes I_d] (w(t) -\alpha \nabla F(z(t))){\|_{\pi\otimes 1_d}.}
\end{split}
\end{equation}
We further manipulate this identity to have
\begin{equation*}
\begin{split}
&\|w(t+1) -(W^{\infty}\otimes I_d) w(t+1) \|_{\pi\otimes 1_d}
\\
&= \|[(W-W^{\infty} )\otimes I_d] (w(t) -\alpha \nabla F(z(t)))\|_{\pi\otimes 1_d}\\
&= \|[(W-W^{\infty} )\otimes I_d] [((I_n-W^{\infty} )\otimes I_d)w(t) -\alpha \nabla F(z(t))]\|_{\pi\otimes 1_d},
\end{split}
\end{equation*}
where we have used $W^{\infty} W  =W  W^{\infty} =(W^{\infty} )^2 = W^{\infty}$ in the first and the second equalities. By the triangle inequality {and Lemma \ref{lem-2-2},} we estimate \eqref{eq-3-10} as follows:
\begin{equation}\label{eq-3-0}
\begin{split}
&\|w(t+1) -(W^{\infty}\otimes I_d) w(t+1) \|_{\pi\otimes 1_d}
\\
& \leq \vertiii{(W -W^{\infty} )\otimes I_d}_{\pi\otimes1_d}\left[\|((I_n-W^{\infty} )\otimes I_d)w(t) \|_{\pi\otimes1_d}+\alpha \|\nabla F(z(t))\|_{\pi\otimes 1_d}\right]
\\
&= \vertiii{W -W^{\infty} }_{\pi}\left[\|w(t) -(W^{\infty} \otimes I_d)w(t) \|_{\pi\otimes1_d}+\alpha \|\nabla F(z(t))\|_{\pi\otimes 1_d}\right].
\end{split}
\end{equation} 
This together with Lemma \ref{lem-2-1} gives
\begin{equation}\label{eq-3-12}
\begin{split}
&\|w(t+1) -(W^{\infty}\otimes I_d)w(t+1) \|_{\pi\otimes 1_d} 
\\
&\leq \sigma_W  \|w(t) - (W^{\infty}\otimes I_d) w(t)\|_{\pi\otimes 1_d} +\alpha  \sigma_W \|\nabla F(z(t))\|_{\pi\otimes 1_d}.
\end{split}
\end{equation}
To estimate $\|\nabla F(z(t))\|_{\pi\otimes 1_d}$  further, we write
\begin{equation}\label{eq-3-7}
\begin{split}
\nabla F(z(t))
& = (\nabla F ( z(t) ) - \nabla F(1_n\otimes\bar{w}(t)))
\\
&\qquad + (\nabla F(1_n\otimes\bar{w}(t)) - \nabla F(1_n\otimes w_*)) + \nabla F(1_n\otimes w_*).
\end{split}
\end{equation}
Using Lemma \ref{lem-2-5} we have
\begin{equation*}
\begin{split}
& \| z(t) - 1_n\otimes\bar{w}(t) \|_{\pi\otimes 1_d}
\\
&\leq \delta \|w(t) - n\pi \otimes \bar{w}(t)\|_{\pi\otimes1_d} + \delta {\|1_n-n\pi\|_\pi \sigma_W^{t} }\Big( \|\bar{w}(t)-w_*\|+\|w_*\|\Big). 
\end{split}
\end{equation*}
Combining this with \eqref{eq-3-7} yields
\begin{equation*}
\begin{split}
&\|\nabla F(z(t))\|_{\pi\otimes1_d}
\\
& \leq  L \| z(t) - 1_n \otimes \bar{w}(t) \|_{\pi\otimes1_d} + L\|1_n\otimes\bar{w}(t) - 1_n\otimes w_*\|_{\pi\otimes1_d} +\|\nabla F(1_n\otimes w_*)\|_{\pi\otimes1_d}
\\
&\leq L \delta \|w(t) -n \pi \otimes \bar{w}(t)\|_{\pi\otimes1_d} + L \Big( \delta {\|1_n-n\pi\|_\pi \sigma_W^{t} }+ \sqrt{ \sum_{j=1}^n \frac{1}{\pi_j}} \Big) \|\bar{w}(t) -w_*\| 
\\
&\qquad + \|\nabla F (1_n\otimes w_*)\|_{\pi\otimes1_d} + L\delta {\|1_n-n\pi\|_\pi \sigma_W^{t} }\|w_*\|.
\end{split}
\end{equation*}
Putting this estimate in \eqref{eq-3-12} gives the following estimate
\begin{equation*}
\begin{split}
&\|w(t+1) -(W^{\infty}\otimes I_d) w(t+1) \|_{\pi\otimes 1_d} 
\\ 
&\leq  \sigma_W (1+\alpha L\delta)\|(W^{\infty} \otimes I_d) w(t) - w(t) \|_{\pi \otimes 1_d} 
\\
&\quad + \alpha L\sigma_W   \Big(\delta {\|1_n-n\pi\|_\pi \sigma_W^{t} } + \sqrt{\sum_{j=1}^n \frac{1}{\pi_j}}\Big) \|\bar{w}(t) - w_*\|\\
&\quad+\alpha\sigma_W \Big( L \delta {\|1_n-n\pi\|_\pi \sigma_W^{t} }   \|w_*\|+\|\nabla F(1_n\otimes w_*)\|_{\pi\otimes1_d}\Big).
\end{split}
\end{equation*} 
The proof is done.
\end{proof}

\section{Estimates on the distance between the average and the optimizer}\label{sec-5}
In this section, we find an estimate of $\|\bar{w}(t+1) -w_*\|$ in terms of $\|\bar{w}(t) - w_*\|$ and $\|w(t) - (W^{\infty}\otimes I_d) w(t)\|_{\pi \otimes 1_d}$. We begin {by} recalling a classical result in the following lemma (see e.g. \cite{B}).
\begin{lem}\label{lem-2-3}
Assume that $f$ is $\beta$-strongly convex and $L$-smooth. Then 
\begin{equation*}
\langle \nabla f(x)- \nabla f(y),~x-y \rangle \geq \frac{L\beta}{L+\beta}\|x-y\|^2 + \frac{1}{L+\beta}\|\nabla f(x) - \nabla f(y)\|^2
\end{equation*}
holds for all $x,y\in \mathbb{R}^d$. 
\end{lem}  
Now we state the main result of this section.
\begin{prop}\label{lem-4-1}
{Assume that (F1) and (F2) hold, and  $f$ is $\beta$-strongly convex, and that $\kappa(t)\geq1$ for all $t\geq0$. Let $\gamma  =  \frac{ \beta L}{L+\beta}\in (0,1)$. Then, for  $\alpha\in \Big(0, \frac{2}{L+\beta}\Big]$,} we have the following estimates
\begin{equation}\label{eq-3-2}
\begin{split}
\|\bar{w}(t+1) - w_*\|& \leq \Big(1-\gamma \alpha + \frac{\alpha L\delta}{n} {\|1_n-n\pi\|_\pi \sigma_W^{t} }\Big) \|\bar{w}(t) -w_*\| 
\\
&\quad  + \frac{\alpha L \delta}{n} \|(W^{\infty}\otimes I_d) w(t) -w(t)\|_{ \pi\otimes 1_d} + \frac{\alpha L \delta}{n}{\|1_n-n\pi\|_\pi \sigma_W^{t} }  \|w_*\|.
\end{split}
\end{equation}
where $w_*$ is the minimizer of the problem \eqref{goal}. 
\end{prop}
\begin{proof}  From \eqref{eq-2-2} we find
\begin{equation*}
\begin{split}
\bar{w}(t+1) &  = {\frac{1}{n}(1_n^\top\otimes I_d)w(t+1)=\frac{1}{n}(1_n^\top\otimes I_d)(W \otimes I_d)x(t)=\frac{1}{n}(1_n^\top  \otimes I_d)x(t)}\\
&=\frac{1}{n}(1_n^\top  \otimes I_d)(w(t)-\alpha\nabla F(z(t)))
\\
&= \bar{w}(t) -  \frac{\alpha}{n}\sum_{i=1}^{n} \nabla f_i (z_i (t))
\\
& = \bar{w}(t) -  \frac{\alpha}{n}\sum_{i=1}^{n} \nabla f_i \left(\frac{w_i (t)}{y_i (t)}\right).
\end{split}
\end{equation*}
Let us write the above term as
\begin{equation}\label{eq-3-3}
\begin{split}
\bar{w}(t+1)-w_* & = \bar{w}(t) -w_* - \frac{\alpha}{n} \sum_{i=1}^{n} \nabla f_i (\bar{w}(t)) + \frac{\alpha}{n} \sum_{i=1}^{n} \left( \nabla f_i (\bar{w}(t))-\nabla f_i \left(\frac{w_i (t)}{y_i (t)}\right)\right).
\end{split}
\end{equation}
Using Lemma \ref{lem-2-3}, we have
\begin{equation*}
\begin{split}
&\|\bar{w}(t) - w_* - \alpha \nabla f(\bar{w}(t))\|^2
\\
& = \|\bar{w}(t) - w_* \|^2 - 2\alpha \langle \bar{w}(t) - w_*,~\nabla f(\bar{w}(t)) \rangle + \alpha^2 \|\nabla f(\bar{w}(t))\|^2
\\
&\leq \|\bar{w}(t) -w_*\|^2 - \frac{2\alpha \beta L}{L+\beta} \|\bar{w}(t) -w_*\|^2 - \frac{2\alpha}{L+\beta} \|\nabla f(\bar{w}(t))\|^2 + \alpha^2 \|\nabla f(\bar{w}(t))\|^2.
\end{split}
\end{equation*}
Since $\alpha \leq \frac{2}{L+\beta}$, the above inequality gives
\begin{equation*}
\begin{split}
\|\bar{w}(t) - w_* - \alpha \nabla f(\bar{w}(t))\|^2&\leq \Big(1 - \frac{2\alpha \beta L}{L+\beta}\Big) \|\bar{w}(t) -w_*\|^2
\\
&\leq \Big( 1- \frac{\alpha \beta L}{L+\beta}\Big)^2  \|\bar{w}(t) -w_*\|^2.
\end{split}
\end{equation*}
Hence we have
\begin{equation}\label{eq-3-4}
\left\| \bar{w}(t) -w_* - \frac{\alpha}{n} \sum_{i=1}^{n} \nabla f_i (\bar{w}(t))  \right\| \leq (1- \gamma \alpha) \|\bar{w}(t) -w_*\|,
\end{equation}
where $\gamma  =   \frac{ \beta L}{L+\beta}\in (0,1)$. In order to esimate the right-most term of \eqref{eq-3-3}, we use \eqref{eq-2-4} to have
\begin{equation}\label{eq-3-5}
\left\| \frac{\alpha}{n} \sum_{i=1}^{n} \left( \nabla f_i (\bar{w}(t))-\nabla f_i \left(\frac{w_i (t)}{y_i (t)}\right)\right)\right\| \leq \frac{\alpha L}{n}   \sum_{i=1}^{n} \left\|\bar{w}(t) -\frac{w_i (t)}{y_i (t)}\right\|,
\end{equation}
whose {right-hand} side is estimated as
\begin{equation*}
\begin{split}
\sum_{i=1}^n \Big\|\bar{w}(t) - \frac{w_i (t)}{y_i (t)}\Big\|& \leq \Big( \sum_{i=1}^n \pi_i \Big)^{1/2} \Big( \sum_{i=1}^n \frac{1}{\pi_i} \Big\| \bar{w}(t) - \frac{w_i (t)}{y_i (t)}\Big\|^2\Big)^{1/2}
\\
& =  \| 1_n\otimes \bar{w}(t) - z(t) \|_{\pi\otimes1_d}.
\end{split}
\end{equation*}
Gathering the estimates \eqref{eq-3-4} and  \eqref{eq-3-5} in \eqref{eq-3-3} we find
\begin{equation*}
\|\bar{w}(t+1) -w_*\| \leq (1-\gamma \alpha) \|\bar{w}(t) -w_*\| + \frac{\alpha L}{n}  \| 1_n\otimes\bar{w}(t) - z(t) \|_{\pi\otimes 1_d}.
\end{equation*}
Combining this with Lemma \ref{lem-2-5} gives  the following estimate
\begin{equation*}
\begin{split}
\|\bar{w}(t+1) - w_*\|& \leq \Big(1-\gamma \alpha + \frac{\alpha L\delta}{n} {\|1_n-n\pi\|_\pi \sigma_W^{t} }\Big) \|\bar{w}(t) -w_*\| 
\\
&\quad  + \frac{\alpha L \delta}{n} \|(W^{\infty}\otimes I_d) w(t) -w(t)\|_{ \pi\otimes 1_d} + \frac{\alpha L \delta}{n}{\|1_n-n\pi\|_\pi \sigma_W^{t} }  \|w_*\|.
\end{split}
\end{equation*}
The proof is done.
\end{proof}

\section{Uniform boundedness of the sequence}\label{sec-6}
Now we shall use the recursive estimates of the previous sections to establish the uniform boundedness property of the sequence $\{w_k (t)\}$. For the proof, we recall the estimates of the sequences $A_t = \|\bar{w}(t) - w_*\|$ and $B_t = \|(W^{\infty}\otimes I_d) w(t) -w(t)\|_{ \pi\otimes 1_d}$ given by Propositions \ref{lem-3-2} and \ref{lem-4-1}:
\begin{equation}\label{eq-5-1}
A_{t+1} \leq (1-\gamma \alpha + \frac{\alpha L\delta}{n} {\|1_n-n\pi\|_\pi \sigma_W^{t} }) A_t + \frac{\alpha L\delta}{n} B_t + \frac{\alpha L\delta}{n}{\|1_n-n\pi\|_\pi \sigma_W^{t} } {\|w_*\|,}
\end{equation}
and
\begin{equation}\label{eq-5-2}
\begin{split}
B_{t+1}& \leq \sigma_W (1+\alpha L\delta) B_t + \sigma_W  (\alpha L) \textbf{D}_1 A_t  +\alpha \sigma_W  \textbf{D}_2,
\end{split}
\end{equation}
where $\textbf{D}_1 = D_1[0]$ and $\textbf{D}_2 = D_2[0]$.
Now we establish the uniform boundedness result of the sequence $\{(A_t, B_t)\}_{t \in \mathbb{N}_0}$. 

\begin{proof}[Proof of Theorem \ref{thm-2-10}]
Take a value $t_0 \in \mathbb{N}$ such that for $t \geq t_0$ we have\\ $ (L\delta/n) {\|1_n-n\pi\|_\pi \sigma_W^{t} } \leq \gamma   /2$.
Choose $R>0$ as
\begin{equation*}
R = \max\Big\{ A_{t_0},~\frac{B_{t_0}}{b}, ~2\|w_*\|,~ \frac{\sigma_W \alpha \textbf{D}_2}{b(1-\sigma_W) - \alpha L \sigma_W (\delta b +\delta {\|1_n-n\pi\|_\pi   } + \sqrt{\sum_{j=1}^n \frac{1}{\pi_j}})}\Big\}.
\end{equation*}
By an inductive argument, we will show that 
\begin{equation}\label{eq-5-3}
A_t \leq R\quad \textrm{and}\quad B_t \leq bR
\end{equation}
for all $t \geq t_0$.
 Assume that $A_t \leq R$ and $B_t \leq bR$ for some $t \geq t_0$. Then we use \eqref{eq-5-1} to deduce
\begin{equation*}
\begin{split}
A_{t+1}& \leq (1-\gamma \alpha /2) R + \frac{\alpha L\delta}{n}b R +\frac{\alpha L\delta}{n}{\|1_n-n\pi\|_\pi \sigma_W^{t} } \|w_*\|
\\
& = \Big(1 -\Big(\gamma/2 - (L\delta/ n)b\Big)\alpha\Big) R + \frac{\gamma \alpha}{2} \cdot \frac{R}{2}
\\
& = R.
\end{split}
\end{equation*}
Next we estimate $B_{t+1}$ using \eqref{eq-5-2} as
\begin{equation*}
\begin{split}
B_{t+1} & \leq \sigma_W (1+\alpha L \delta) b R + \sigma_W \alpha L \textbf{D}_1 R + \sigma_W \alpha \textbf{D}_2
\\
& \leq \sigma_W (1+\alpha L \delta) b R + \sigma_W \alpha L \Big(\delta {\|1_n-n\pi\|_\pi  } + \sqrt{\sum_{j=1}^n \frac{1}{\pi_j}}\Big) R + \sigma_W \alpha \textbf{D}_2
\\
&\leq bR.
\end{split}
\end{equation*}
The proof is done.
\end{proof} 

\section{Convergence estimates}\label{sec-7}
In this section, we prove the main convergence result of the gradient-push algorithm stated in Theorem \ref{thm-2-5}. For this aim, we prepare the following two basic lemmas to analyze further the sequential estimates \eqref{eq-5-1} and \eqref{eq-5-2} of $A_t$ and $B_t$. 
\begin{lem}\label{lem-5-2}
Assume that a sequence $\{Y_t\}_{t\geq 0} \subset \mathbb{R}$ satisfies
\begin{equation}\label{eq-5-10}
Y_{t+1} \leq (1-\alpha )Y_t +C\quad \forall t \geq 0
\end{equation}
for some $C>0$ and $\alpha \in (0,1)$. Then we have
\begin{equation*}
Y_t \leq (1-\alpha)^t Y_0 + \frac{C}{\alpha}.
\end{equation*}
\end{lem}
\begin{proof}
Using \eqref{eq-5-10} iteratively, we have
\begin{equation*}
\begin{split}
Y_t &\leq (1-\alpha)^t Y_0 + C \sum_{k=0}^{t-1} (1-\alpha)^k.
\\
&\leq (1-\alpha)^t Y_0 + \frac{C}{\alpha},
\end{split}
\end{equation*}
which proves the lemma.
\end{proof}
\begin{lem}\label{lem-5-3}
Assume that a sequence $\{Y_t\}_{t \geq 0} \subset \mathbb{R}$ satisfies
\begin{equation}\label{eq-6-1}
Y_{t+1} \leq (1-\alpha) Y_t + C \rho^t \quad \forall t  \geq 0 
\end{equation}
for some $C>0$, $\alpha \in (0,1)$ and $\rho \in (0,1)$.
Then we have
\begin{equation*}
Y_t \leq (1-\alpha)^t Y_0 + Ct \left(\max\{1-\alpha, \rho\}\right)^{t-1}.
\end{equation*}
\end{lem}
\begin{proof}
Making use of \eqref{eq-6-1} recursively, we find
\begin{align*}
Y_t &\leq (1-\alpha)^t Y_0 + C\sum_{k=0}^{t-1} (1-\alpha)^k \rho^{t-1-k}\\
&\leq (1-\alpha)^t Y_0 + C\sum_{k=0}^{t-1} \left(\max\{1-\alpha, \rho\}\right)^{t-1}\\
&=(1-\alpha)^t Y_0 + Ct \left(\max\{1-\alpha, \rho\}\right)^{t-1}.
\end{align*}
The proof is done.
\end{proof}
Now we are ready to establish the main convergence result.
\begin{proof}[Proof of Theorem \ref{thm-2-5}]
Let $R:=\sup_{t\geq0}A_t$. By \eqref{eq-5-2}, we get
\begin{equation*}
\begin{split}
B_{t+1} &\leq \sigma_W (1+\alpha L) B_t + \sigma_W (\alpha L) \textbf{D}_1 R  +\alpha \sigma_W \textbf{D}_2
\\
&\leq \rho B_t + \alpha \sigma_W \bar{R}\quad \textrm{for all}~t \geq 0,
\end{split}
\end{equation*}
where we have let $\rho = \sigma_W (1+\alpha L)$ and $\bar{R} = L\textbf{D}_1 R +\textbf{D}_2$. Using this we obtain
\begin{equation*}
\begin{split}
B_t & \leq \rho^t B_0 + \Big( \rho^{t-1} + \rho^{t-2} +\cdots + 1\Big) \alpha \sigma_W \bar{R}
\\
&\leq \rho^t B_0 + \frac{\alpha \sigma_W \bar{R}}{1-\rho},
\end{split}
\end{equation*}
which proves \eqref{eq-2-16}. Using this inequality in \eqref{eq-5-1} we have
\begin{equation*}
\begin{split}
A_{t+1} & \leq (1-\gamma \alpha ) A_t + \frac{\alpha L\delta}{n} \rho^t B_0 + \alpha^2 \frac{L \delta}{n} \frac{\sigma_W \bar{R}}{1-\rho} + \frac{\alpha L \delta}{n} {\|1_n-n\pi\|_\pi \sigma_W^{t} } (\|w_*\|+R).
\end{split}
\end{equation*}
To analyze this inequality, we consider three sequences $A_t^1, A_t^2, A_t^3$ such that for $t \geq 0$,
\begin{equation*}
\begin{split}
A_{t+1}^1 & = (1-\gamma \alpha ) A_t^1   + \alpha^2 \frac{L\delta}{n} \frac{\sigma_W \bar{R}}{1-\rho}
\\
A_{t+1}^2  &= (1-\gamma \alpha ) A_t^2 + \frac{\alpha L\delta}{n} \rho^t B_0
\\
A_{t+1}^3 &= (1-\gamma \alpha ) A_t^3 + \frac{\alpha L\delta}{n} {\|1_n-n\pi\|_\pi \sigma_W^{t} }  (\|w_*\|+R),
\end{split}
\end{equation*}
with initial values $A_0^1 = A_0$, $A_0^2 =0$, and $A_0^3 =0$. 
We use Lemma \ref{lem-5-2} and Lemma \ref{lem-5-3} to analyze these sequences to  find
\begin{equation*}
\begin{split}
A_t^1 &\leq (1-\gamma \alpha )^tA_0 + \frac{\alpha L\delta}{n\gamma} \frac{\sigma_W \bar{R}}{1-\rho}
\\
A_t^2 &\leq \frac{\alpha L\delta}{n} B_0 t\Big[ \max\{1-\gamma \alpha, \rho\}\Big]^{t-1}
\\
A_t^3 & \leq \frac{\alpha L \delta}{n} {\|1_n-n\pi\|_\pi   } (\|w_*\|+R) t\Big[ \max\{1-\gamma \alpha, e^{-c}\}\Big]^{t-1}.
\end{split}
\end{equation*}
Collecting the above estimates, we get
\begin{equation*}
\begin{split}
A_t & \leq (1-\gamma \alpha )^t A_0+ \frac{\alpha L\delta}{n} B_0 t\Big[ \max\{1-\gamma \alpha, \rho\}\Big]^{t-1}
\\
&\quad + \frac{\alpha L \delta}{n} {\|1_n-n\pi\|_\pi  }  (\|w_*\|+R) t\Big[ \max\{1-\gamma \alpha, e^{-c}\}\Big]^{t-1} + \frac{\alpha L\delta}{n\gamma} \frac{\sigma_W \bar{R}}{1-\rho}.
\end{split}
\end{equation*}
The proof is done.
\end{proof}

\section{Numerical experiments}\label{sec-8}

 Here we provide numerical results for the gradient-push algorithm \eqref{eq-2-1} that support our theoretical convergence result. For all experiments, we set $x_i(0)=1_d$ $(i=1,\dots,n)$,  and generated a random directed graph $G$ with $n$ vertices by letting each possible arc to occur independently {with probability $p=0.7$,} except for the self-loops, which must exist with probability 1. Then the matrix $W$ would be naturally determined by \eqref{eq-2-0}. In the first example, we consider the least square parameter estimation problem where each {local} cost is convex. In the second example, we consider local cost of the form $f_j (x) = \frac{1}{2} x^T A_j x + B_j^T x$ where $A_j$ may not be a positive semidefinite matrix. In the last example, we compare the performance of the gradient-push algorithm with the Push-DIGing algorithm and their hybrid.
 
\subsection{Convex local cost case} We set $d=5, n=10$ and consider the least square parameter estimation problem
\begin{equation*}
f(x) = \frac{1}{2n} \sum_{j=1}^n \|A_j x - b_j\|^2.
\end{equation*} 
Here $A_j$ and $b_j$ are $3\times d$, $3\times1$ matrices respectively, with i.i.d. entries drawn from the standard normal distribution. It is well known that the minimizer of $f$ is given by $x^*=\left(\sum\limits_{j=1}^n A_j^\top A_j\right)^{-1}\sum\limits_{j=1}^n A_j^\top b_j$. Figure \ref{fig1} shows the graph of $\sum\limits_{i=1}^n\|z_i(t)-x^*\|^2$ with respect to the iteration $t$, for $\alpha=10^{-1}, 10^{-2}, 10^{-3}, 10^{-4}$. The limiting errors are proportional to the value of $\alpha$, which supports the result of Theorem \ref{thm-2-5}.

\begin{figure}[htbp]
\includegraphics[height=6cm, width=8.5cm]{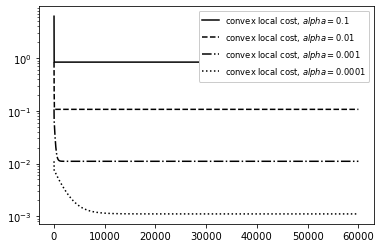}  
\caption{Performance of \eqref{eq-2-2} for the convex local cost case.}
\label{fig1}
\vspace{-0.3cm}
\end{figure}

\subsection{Nonconvex local cost case}

Here we set $d=5, n=10$ and consider the following problem
\begin{equation*}
f(x) = \frac{1}{n} \sum_{j=1}^n f_j (x),
\end{equation*}
where $f_j (x) = \frac{1}{2} x^T (I+C_j) x + d_j^T x$. Here $C_j$ is a symmetric $d\times d$ matrix and $d_j \in \mathbb{R}^d$ for each $1\leq j \leq n$, with i.i.d. entries(excluding the pairs that must be equal because of the symmetry condition) drawn from the standard normal distribution. We consider the case that the aggregate cost $f$ satisfies the assumption (F2). If the resulting $f$ is not convex, then all random entries are sampled again. The minimizer of $f$ is given by $x^*=-\left(\sum_{j=1}^N(I+C_j)\right)^{-1}\sum_{j=1}^Nd_j$. Figure \ref{fig2} shows the graph of $\sum\limits_{i=1}^n\|z_i(t)-x^*\|^2$ with respect to the iteration $t$, for $\alpha=10^{-1}, 10^{-2}, 10^{-3}, 10^{-4}$. Although some $f_j$ may not be convex, the limiting error seems to be proportional to the value of $\alpha$, as expected in the result of Theorem \ref{thm-2-5}.

\begin{figure}[htbp]
\includegraphics[height=6cm, width=8.5cm]{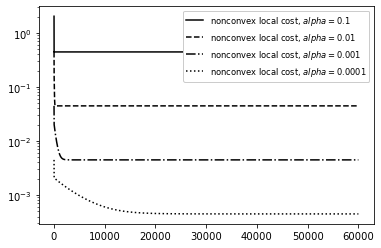}  
\caption{Performance of \eqref{eq-2-2} for the nonconvex local cost case.}
\label{fig2}
\vspace{-0.3cm}
\end{figure}

\subsection{\textcolor{black}{Comparison between the gradient-push, Push-DIGing, and their hybrid algorithm}}

In this subsection, we compare the performance of the gradient-push algorithm with the Push-DIGing algorithm and their hybrid. The Push-DIGing algorithm is described as follows:
\begin{equation}\label{eq-1-20}
\begin{split}
\mathbf{x}_i (t+1) & = \sum_{j=1}^n W_{ij}  \mathbf{x}_j (t) - \alpha  {v}_i (t)
 \\
 {y}_i (t+1)& =\sum_{j=1}^n W_{ij}  {y}_j (t)
 \\
 {z}_i (t+1) & = \frac{\mathbf{x}_i (t+1)}{ {y}_i (t+1)}
 \\
 {v}_i (t+1)& = \sum_{j=1}^n W_{ij} {v}_j (t) + \nabla f_i ( {z}_i (t+1))- \nabla f_i ( {z}_i (t)),
 \end{split}
 \end{equation}
 where $ {v}_i (0) = \nabla f_i ( {z}_i (0))$. It was proved in the works \cite{XXK, NOS} that the Push-DIGing algorithm converges linearly to the minimizer, given that the constant  stepsize $\alpha >0$ is less than a specific value. 

We set $n=50$, $d=10$, and consider the cost functions $f_j :\mathbb{R}^d \rightarrow \mathbb{R}$ given as
\begin{equation*}
f_j (x) = x^\top (I+Z_j) x + z_j^\top x\quad (j=1,\dots, n-1),\qquad f_n (x) = \frac{1}{2} x^T ((n+2)I+2Z_n) x + z_n^\top x,
\end{equation*}
where the entries of $Z_j \in \mathbb{R}^{d\times d} $, $z_j \in \mathbb{R}^d$ $(j=1,\dots,n)$ are independently drawn from the standard normal distribution.
The minimizer of $f$ is given by 
\[
x^*=-\left(\sum_{j=1}^{n}(Z_j+Z_j^\top)+3nI\right)^{-1}\sum_{j=1}^Nz_j. 
\]

First, we test the performance of the gradient-push and Push-DIGIng algorithms with various stepsizes. The left side of Figure 3 shows the performance of the gradient-push algorithm with stepsizes $\alpha\in\{0.012, 0.014, 0.016, 0.018, 0.02\}$ and the right side represents the  performance of the Push-DIGing algorithm with stepsizes $\alpha \in\{0.0072, 0.0074, 0.0076$, $0.0078$, $0.008\}$.

\begin{figure}[htbp]
\includegraphics[height=5cm, width=7cm]{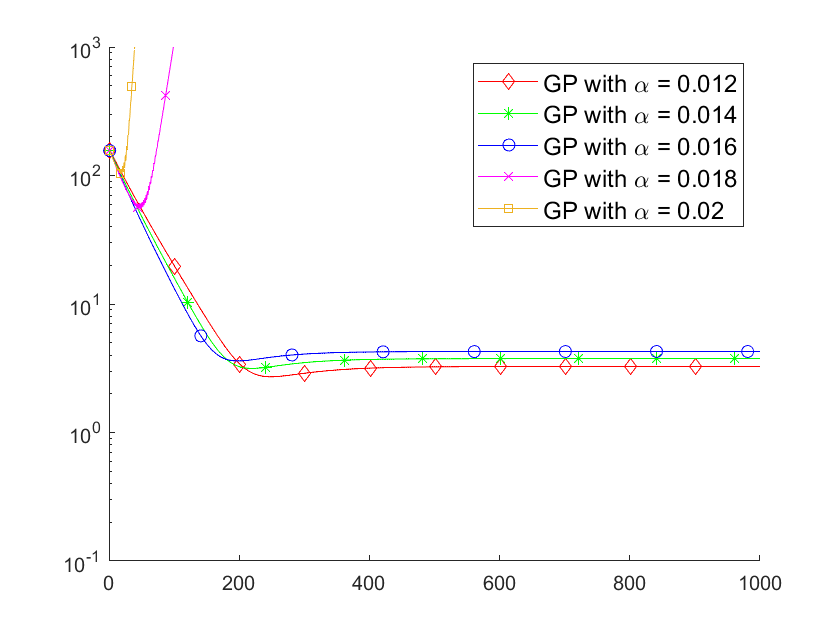}  
\includegraphics[height=5cm, width=7cm]{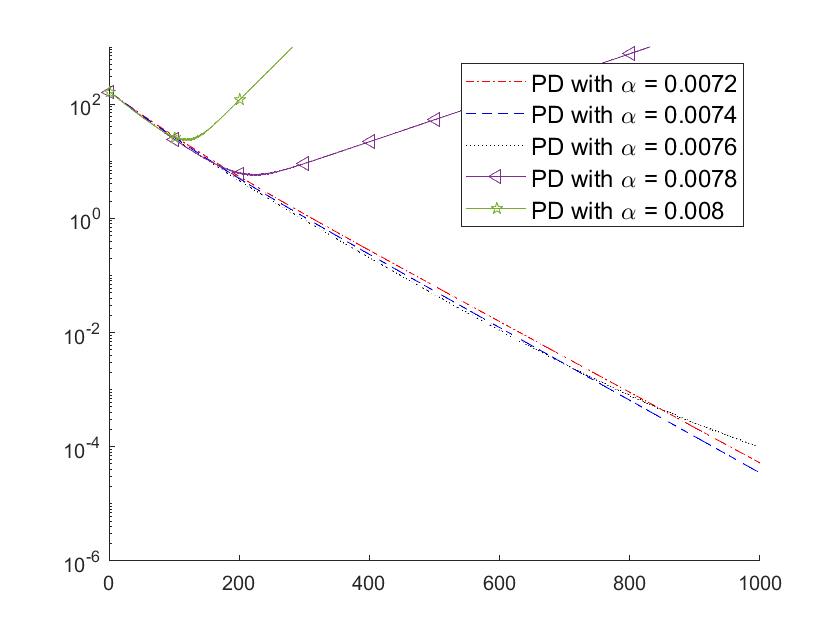}  
\caption{Graphs of $\sum\limits_{i=1}^n\|z_i(t)-x^*\|^2$ with respect to the iteration $t$, for the gradient-push (GP) and the Push-DIGing (PD) algorithms with various stepsizes.}
\label{fig2}
\vspace{-0.3cm}
\end{figure}

We observe that the gradient-push algorithm performs well for $\alpha\in \{0.012, 0.014, 0.016\}$, while it diverges for $\alpha\in\{0.018, 0.02\}$. Also, Push-DIGing converges linearly for $\alpha\in\{0.0072, 0.0074, 0.0076\}$ and diverges for $\alpha\in\{0.0078,  0.008\}$. It turns out that the gradient-push algorithm converges faster than the Push-DIGing algorithm up to a certain number of iterations, while the Push-DIGing algorithm converges to the minimizer exactly. This motivates us to consider the following hybrid algorithm that performs the gradient-push first and then switches to the Push-DIGing algorithm \cite{CCK}. 
\begin{algorithm}
    \begin{algorithmic}[1] 
    \caption{Hybrid algorithm}\label{algo}
    \REQUIRE Initial $x_i(0) \in \mathbb{R}^d$, $y_i(0)=1$ for all $i\in\{1,\cdots,m\}$. Number of iterates $GP_{iterate} \in \mathbb{N}$ and $Total_{iterate} \in \mathbb{N}$.
            \FOR{$k =1 \ \textrm{to} \ GP_{iterate}$,}
                
                \STATE perform the gradient-push algorithm \eqref{eq-2-1} with $\alpha =\alpha_{0}$. 
                  
            \ENDFOR 
            
            \STATE  Set $\mathbf{x}_i (0) = w_i (GP_{iterate})$, $\mathbf{y}_i (0)=y_i (GP_{iterate})$, $\mathbf{z}_i (0)=z_i (GP_{iterate})$, $\mathbf{v}_i (0) =\nabla f(z_i(GP_{iterate}))$.
            
            \FOR{$k=GP_{iterate}+1 \ \textrm{to} \ Total_{iterate}$,}
                
                \STATE  perform the Push-DIGing algorithm \eqref{eq-1-20} with $\alpha =\alpha_1$.
                  
            \ENDFOR 
    \end{algorithmic}
    \end{algorithm}
We compare the performance of the gradient-push algorithm with stepsizes   $\alpha\in \{0.012, 0.014, 0.016\}$, the Push-DIGing algorithm with stepsizes $\alpha \in \{0.0072, 0.0074, 0.0076\}$, and Algorithm 1 with stepsize $(\alpha_0, \alpha_1) = (0.016, 0.0074)$ and $GP_{iterate}=100$. We measure performance with respect to both time iteration and communication cost. The communication cost counts $(d+1)$ per each iteration  for the gradient-push algorithm while $(2d+1)$ per each iteration for the Push-DIGing algorithm since the gradient-push communicates $(x_j (t), y_j (t)) \in \mathbb{R}^{d +1}$ while the Push-DIging algorithm communicates $(\mathbf{x}_j (t), y_j (t), v_j (t)) \in \mathbb{R}^{2d+1}$ at each iteration. 

The left and right sides of Figure 4 show the graphs of the algorithms with repect to the iterations and  the communication costs respectively. It reveals that the advantage of the gradient-push algorithm for converging fast to a neighborhood of the minimizer is more apparent when we concern the communication costs. It also shows that the hybrid algorithm performs best as it takes advantages of the two algorithms.

\begin{figure}[htbp]
\includegraphics[height=5cm, width=7cm]{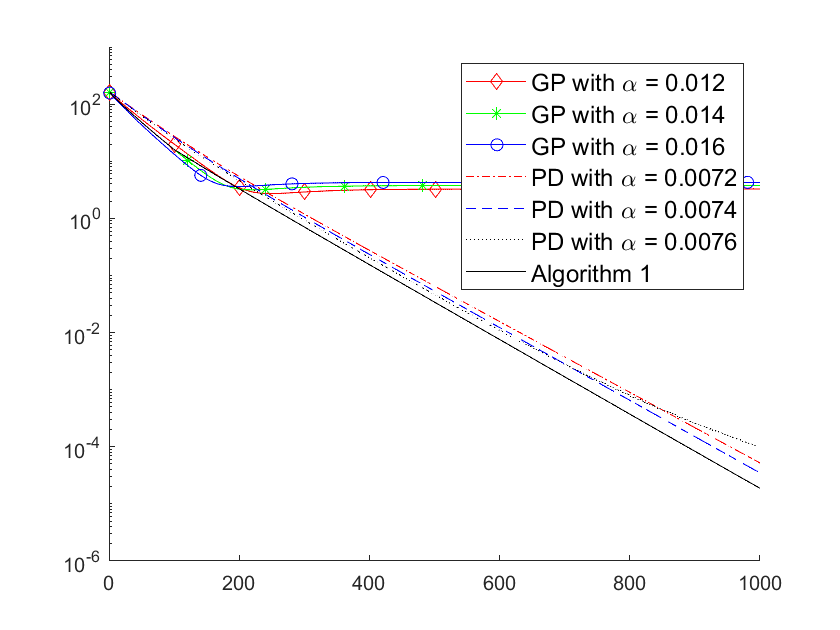}  
\includegraphics[height=5cm, width=7cm]{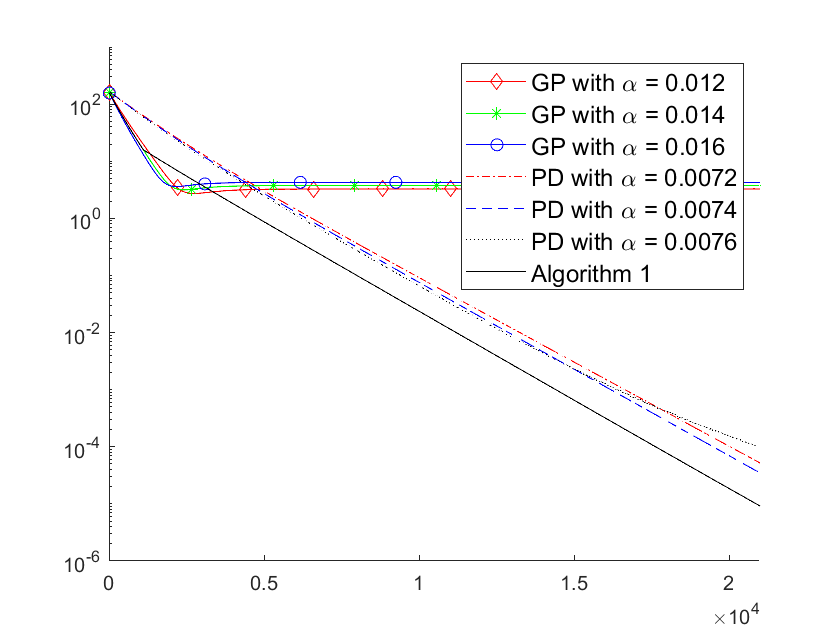}  
\caption{Graphs of $\sum\limits_{i=1}^n\|z_i(t)-x^*\|^2$ with respect to time iteration (left) and communication cost (right) for the gradient-push(GP), Push-DIGing(PD) algorithms, and Algorithm 1.}
\label{fig2}
\vspace{-0.3cm}
\end{figure}




{
\section*{Acknowledgments}
The work of Woocheol Choi was supported by the National Research Foundation of Korea(NRF) grant funded by the Korea government(MSIT)  (No. RS-2024-00336077). The work of Doheon Kim was supported by the National Research Foundation of Korea(NRF) grant funded by the Korea government(MSIT) (No. RS-2022-00166699). Doheon Kim is grateful for support by the Open KIAS Center at Korea Institute for Advanced Study. The work of Seok-Bae Yun was supported by the National Research Foundation of Korea(NRF) grant funded by the Korea government(MSIT) (No. NRF-2023R1A2C1005737). 
}
{
\section*{Data availability}
The data generated during the current study are available from the corresponding author on reasonable request.
}
\appendix

\section{An intuition behind the gradient-push algorithm}
In this section, we give an intuition behind the gradient-push \mbox{algorithm \eqref{eq-2-2}.} 
	We may {rewrite} the algorithm \eqref{eq-2-2} in terms of $w(t)$ as
	\begin{equation}\label{eq-2-37}
		w(t+1) = (W\otimes I_d) (w(t) - \alpha \nabla F(z(t)).
	\end{equation}
In order to understand the algorithm \eqref{eq-2-2} in the above form, we need to figure out why we define $z(t)$ as $z_i (t) = w_i (t)/y_i (t)$. First, we prove that regardless of the value of the variable $z(t)$, the variable $w(t)$ in \eqref{eq-2-37} achieves {weighted} consensus up to an $O(\alpha)$ error under a mild boundedness assumption. {To be more precise,} 
we have the following lemma.
\begin{lem}Assume that $\sup_{s \geq 0}  \|\nabla F(z(s))\| < \infty$. Then there exist constants $D>0$ and $\zeta \in (0,1)$ determined by the graph such that
	\begin{equation*}  
		\|w(t)-n\pi \otimes \bar{w}(t)\| 
		\leq  D\zeta^t \|w(0)\| - \frac{(\alpha D)\zeta}{1-\zeta}  \sup_{s \geq 0}  \|\nabla F(z(s))\|
	\end{equation*}
	for all $t \geq 0$.
\end{lem}
\begin{rem}
The boundedness assumption in the above lemma is actually guaranteed by the result of Theorem \ref{thm-2-10} under an assumption on the stepsize $\alpha >0$.
\end{rem}
\begin{proof}
	Using \eqref{eq-2-37} iteratively, we have
	\begin{equation}\label{eq-2-33}
		w(t) = W^t \otimes I_d (w(0)) - \alpha \sum_{k=1}^{t-1} \Big( W^{t-k}\otimes I_d\Big) \nabla F(z(k)).
	\end{equation}
	Multiplying $\pi 1_n^T$ in front of the both sides and using the property that $\pi 1_n^T W = \pi 1_n^T$, we get the following {equation:}
	\begin{equation}\label{eq-2-34}
		\pi \otimes \bar{w}(t) = \pi 1_n^T \otimes I_d (w(0)) -\alpha \sum_{k=1}^{t-1} \pi 1_n^T \otimes I_d \nabla F (z(k)).
	\end{equation}
	Combining \eqref{eq-2-33} and \eqref{eq-2-34} gives the following {identity:}
	\begin{equation} \label{eq-2-36}
		w(t)-n\pi \otimes \bar{w}(t) = \Big(W^t - \pi 1_n^T\Big) \otimes I_d (w(0)) - \alpha \sum_{k=1}^{t-1} \Big[ (W^{t-k}-\pi 1_n^T)\otimes I_d\Big] \nabla F(z(k)).
	\end{equation}
	By the property \eqref{eq-2-40}, there exist constants $D>0$ and $\zeta>0$ such that
	\begin{equation*}
		\vertiii{W^t - \pi 1_n^T} \leq D \zeta^t
	\end{equation*}
	for all $t \geq 0$. Inserting this estimate in \eqref{eq-2-36} we deduce
	\begin{equation*}  
		\begin{split}
			&\|w(t)-n\pi \otimes \bar{w}(t)\| 
			\\
			&\hspace{0.7cm}\leq D\zeta^t \|w(0)\| - D \alpha \sum_{k=1}^{t-1} \zeta^{t-k}  \|\nabla F(z(k))\|
			\\
			&\hspace{0.7cm}\leq  D\zeta^t \|w(0)\| - \alpha D \sup_{s \geq 0}  \|\nabla F(z(s))\| \sum_{k=1}^{t-1} \zeta^{t-k}
			\\
			&\hspace{0.7cm} =    D\zeta^t \|w(0)\| - \frac{(\alpha D)\zeta}{1-\zeta}  \sup_{s \geq 0}  \|\nabla F(z(s))\|.
		\end{split}
	\end{equation*}
	The proof is done.
\end{proof}
The above lemma shows that $w(t)/n \pi$ {gets close to $\bar{w}(t)$ as $t \rightarrow \infty$, i.e., achieves consensus,} up to an error $O(\alpha)$. Now, we multiply $\pi 1_n^T$ {to the left} of both sides of \eqref{eq-2-37}, {to} find
\begin{equation*} 
\begin{split}
	\bar{w}(t+1) &=  \bar{w}(t) -\frac{\alpha}{n} \sum_{i=1}^{n} \nabla f_i  (z_i (t))
	\\
	& =  \bar{w}(t) -\alpha \nabla f (\bar{w}(t)) + \frac{\alpha}{n} \sum_{i=1}^{n} \Big( \nabla f_i (\bar{w}(t))- \nabla f_i  (z_i (t))\Big).
	\end{split}
\end{equation*}
From this expression, we see that $\{\bar{w}(t)\}_{t \geq 0}$ follows approximately the gradient descent method of the global cost function $f$ provided that $z_i (t)$ is close enough to $\bar{w}(t)$  for each $1 \leq i \leq n$. On the other hand, the above {lemma} yields that $\frac{w_i (t)}{n \pi_i}$ gets close to $\bar{w}(t)$ up to an $O(\alpha)$-error.

Therefore, if we want  the solution $\bar{w}(t)$ to converge to the minimizer $w_*$ of $f$,  the natural choice of $z_i (t)$ would be ${w_i (t)}/n\pi_i $   if we {know} $\pi_i$ in advance. However, it is often the case that the value $\pi_i$ should be calculated in a distributed manner when a graph of network is given.  Having said this, a natural choice of $z_i (t)$ is given by $z_i (t) = w_i (t)/y_i (t)$ with 
the balancing constants $y(t)$ of \eqref{eq-2-1} in view of the exponential convergence of $y(t)$ to $n\pi$ (see \eqref{eq-2-40}).

\end{document}